\documentclass{amsart}

\usepackage{amssymb}
\usepackage{amsmath}
\usepackage{enumerate}
\usepackage{url}

\usepackage{mathtools}

\DeclarePairedDelimiter\floor{\lfloor}{\rfloor}

\newcommand{\erase}[1]{}

\newtheorem{theorem}{Theorem}[section]
\newtheorem{lemma}[theorem]{Lemma}
\newtheorem{proposition}[theorem]{Proposition}
\newtheorem{corollary}[theorem]{Corollary}

\newtheorem{_algorithm}[theorem]{Algorithm}

\newtheorem{_procedure}[theorem]{Procedure}

\newtheorem{_definition}[theorem]{Definition}
\newenvironment{definition}{\begin{_definition}\rm}{\end{_definition}}

\newtheorem{_remark}[theorem]{\it Remark}
\newenvironment{remark}{\begin{_remark}\rm}{\end{_remark}}

\newtheorem{_example}[theorem]{Example}

\newtheorem{_assumption}[theorem]{Assumption}

\newtheorem{_construction}[theorem]{Construction}

\newtheorem{_claim}[theorem]{Claim}

\newtheorem{_conjecture}[theorem]{Conjecture}

\numberwithin{equation}{section}
\numberwithin{table}{section}
\numberwithin{figure}{section}

\newcommand{\F}{\mathord{\mathbb F}}

\renewcommand{\P}{\mathord{\mathbb  P}}
\newcommand{\Q}{\mathord{\mathbb  Q}}
\newcommand{\R}{\mathord{\mathbb R}}

\newcommand{\Z}{\mathord{\mathbb Z}}

\newcommand{\CCC}{\mathord{\mathcal C}}

\newcommand{\OOO}{\mathord{\mathcal O}}

\newcommand{\QQQ}{\mathord{\mathcal Q}}

\newcommand{\SSS}{\mathord{\mathcal S}}

\newcommand{\SSSS}{\mathord{\mathfrak S}}

\newcommand{\mapdownsurj}{
\hbox{$\bigm\downarrow$}
\llap{\hbox{\raise 2pt\hbox{$\bigm\downarrow$}}}%
\vstrechmapdown
}

\newcommand{\mapupsurj}{
\hbox{$\bigm\uparrow$}
\llap{\hbox{\raise 2pt\hbox{$\bigm\uparrow$}}}%
\vstrechmapup
}

\newcommand{\inj}{\hookrightarrow}
\newcommand{\surj}{\mathbin{\to \hskip -7pt \to}}

\newcommand{\isom}{\xrightarrow{\sim}}

\newcommand{\set}[2]{\{\; {#1} \; \mid \; {#2} \;  \}}
\newcommand{\shortset}[2]{\{ {#1} \,|\, {#2}   \}}

\newcommand{\gen}[1]{\langle {#1}  \rangle}

\newcommand{\tensor}{\otimes}

\newcommand{\sprime}{\sp\prime}

\newcommand{\spar}[1]{\sp{(#1)}}

\newcommand{\sptimes}{\sp{\times}}

\newcommand{\dual}{\sp{\vee}}

\newcommand{\inv}{\sp{-1}}

\newcommand{\GL}{\mathord{\mathrm{GL}}}

\newcommand{\PSL}{\mathord{\mathrm{PSL}}}
\newcommand{\OG}{\mathord{\mathrm{O}}}

\newcommand{\id}{\mathord{\mathrm{id}}}

\newcommand{\Ker}{\operatorname{\mathrm{Ker}}\nolimits}

\newcommand{\Image}{\operatorname{\mathrm{Im}}\nolimits}
\newcommand{\Aut}{\operatorname{\mathrm{Aut}}\nolimits}

\newcommand{\pr}{\mathord{\mathrm{pr}}}

\newcommand{\mystruth}[1]{\phantom{\hbox{\vrule height #1}}}
\newcommand{\mystrutd}[1]{\phantom{\hbox{\vrule depth #1}}}

\newcommand{\intf}[1]{\langle #1\rangle}
\newcommand{\intfvoid}{\intf{\phantom{i}, \phantom{i}}}
\newcommand{\discg}{D}
\newcommand{\discf}{q}
\newcommand{\vv}{\mathord{\textit{\textbf{v}}}}
\newcommand{\vx}{\mathord{\textit{\textbf{x}}}}
\newcommand{\Stab}{\mathord{\rm Stab}}
\newcommand{\str}[1]{\hbox{{\tt "#1"}}}

\begin{document}
\title[An even extremal lattice of rank $64$]%
{An even extremal lattice of rank $64$}
\author{Ichiro Shimada}
\address{Department of Mathematics, 
Graduate School of Science, 
Hiroshima University,
1-3-1 Kagamiyama, 
Higashi-Hiroshima, 
739-8526 JAPAN}
\email{ichiro-shimada@hiroshima-u.ac.jp}

\thanks{This work was supported by JSPS KAKENHI Grant Number 16H03926 and 16K13749.}



\begin{abstract}
We construct an even extremal lattice of rank  $64$
by means of a generalized  quadratic residue code.
\end{abstract}
\maketitle

\section{Introduction}
A \emph{lattice} is a free $\Z$-module $L$ of finite rank
with a symmetric bilinear form
$$
\intfvoid\colon L\times L\to \Z
$$
that makes $L\tensor\R$ a positive-definite real quadratic space.
Let $L$ be a lattice.
The group of automorphisms  of $L$ is denoted by $\OG(L)$.
For simplicity, we write $x^2$ instead of  $\intf{x,x}$ for $x\in L$.
We say that $L$ is \emph{even} (or of type II) if $x^2 \in 2\Z$  holds for all $x\in L$.
In this paper,
we treat only even lattices.
Since $\intfvoid$ is non-degenerate,
the mapping $x\mapsto \intf{x, -}$ embeds $L$ into 
the \emph{dual lattice}
$$
L\dual:=\set{x\in L\tensor\Q}{\intf{x, y}\in \Z\;\; \textrm{for all}\;\; y\in L}.
$$
We say that  $L$ is \emph{unimodular} if this embedding is an isomorphism.
We put
$$
\min (L):=\min\; \shortset{x^2}{ x\in L\setminus \{0\}}.
$$
It is well-known that, if $L$ is an even unimodular lattice, then
its rank $n$ is divisible by $8$ and $\min (L)$ satisfies
\begin{equation}\label{eq:minL}
\min (L)\le 2+2 \floor*{\frac{n}{24}}.
\end{equation}
\begin{definition}
We say that an even unimodular lattice $L$ of rank $n$ is \emph{extremal}
if the equality holds in~\eqref{eq:minL}.
\end{definition}
Extremal lattices are important and interesting,
because they give rise to dense sphere-packings.
%
Extremal lattices of rank $\le 24$ are completely classified.
The famous Leech lattice 
is characterized as the unique (up to isomorphism) extremal lattice of rank  $24$.
On the other hand, the classification of extremal lattices of rank $\ge 32$ seems to be very difficult.
The known examples of extremal lattices
are listed in the website~\cite{NS} administrated by  Nebe and Sloane,
in  Conway and Sloane~\cite[Chapter 1]{CS99},
or in  Gaborit~\cite[Table 3]{Gab2004}.
\par
As is extensively described in Conway and Sloane~\cite{CS99},
there exist various methods of constructing a lattice from a code.
The binary extended quadratic residue codes
play an important role in these constructions.
The most classical examples are that
the extended Hamming code yields
the extremal lattice $E_8$ of rank $8$,
and that the extended Golay code  yields
the Niemeier lattice of type $24 A_1$.
Various generalizations of quadratic residue codes
are investigated.
In particular, in Bonnecaze, Sol\'e and Calderbank~\cite{BPC},
the Leech lattice is constructed by 
a generalized quadratic residue code of length $24$ with components in  $\Z/4\Z$.
See also Chapman and Sol\'e~\cite{CS1996} and Harada and Kitazume~\cite{HK2000}.
\par
In this paper,
we consider a quadratic residue code
with components  in the \emph{discriminant group} $\discg_R:=R\dual/R$
of an even lattice $R$ of small rank,
and construct a lattice $L$ of large rank
as an  even overlattice of the orthogonal direct-sum of copies of $R$
by using the  code as  the gluing data.
As an application, we obtain the following:
\begin{theorem}\label{thm:main}
There exists an extremal lattice $L_{\QQQ}$ of rank $64$
whose automorphism group $\OG(L_{\QQQ})$ is of order $119040$.
This group $\OG(L_{\QQQ})$  contains 
a subgroup $\Gamma_{\QQQ}$  of index $2$ that fits in the exact sequence
\begin{equation}\label{eq:Gammaexact}
0 \;\; \to (\Z/2\Z)^2\;\;  \to\;\; \Gamma_{\QQQ}\;\; \to \;\; \PSL_2(31)\;\; \to \;\; 1.
\end{equation}
\end{theorem}
The code $\QQQ$ that is  used in the construction of $L_{\QQQ}$  is
a generalized quadratic residue code  of length $32$
 with components in the discriminant group  $\discg_R\cong \Z/35\Z$
 of the lattice
$$
R=\left[\begin{array}{cc}  6 & 1 \\ 1 & 6 \end{array}\right].
$$
%
\par
In~\cite{Que84}, Quebbemann constructed 
(possibly several) extremal lattices of rank $64$
as overlattices of the orthogonal direct-sum $E_8(3)^8$
of $8$ copies of $E_8(3)$.
(See also~\cite[Chapter 8.3]{CS99}.)
Here  $E_8(3)$ denotes the lattice obtained from the lattice $E_8$
by multiplying the intersection form by $3$.
We have the following:
\begin{proposition}\label{prop:EEE8}
The lattice $L_{\QQQ}$ does not contain $E_8(3)$ as a sublattice.
\end{proposition}
\begin{corollary}
The lattice $L_{\QQQ}$ cannot be obtained by Quebbemann's construction.
\end{corollary}
In~\cite{N98}, Nebe discovered an extremal lattice 
$$
N_{64}:=L_{8, 2}\otimes L_{32, 2}
$$
of rank $64$, and showed that $\OG(N_{64})$ contains a subgroup
of order $587520$ generated by $6$ elements. (See the website~\cite{NS}.)
Since  $|\OG(L_{\QQQ})|<587520$,
we obtain the following:
\begin{corollary}\label{cor:notisom}
The lattices  $L_{\QQQ}$ and $N_{64}$ are not isomorphic.
\end{corollary}
In Harada, Kitazume and Ozeki~\cite{HKO2002} and Harada and Miezaki~\cite{HM2014},
they also constructed several extremal lattices of rank $64$.
The relation of these lattices with our lattice has not yet been clarified.
\par
We found the lattice $L_{\QQQ}$ by an experimental search.
We hope that several more extremal lattices
can be obtained by the same method.
\par
\medskip
This paper is organized as follows.
In Section~\ref{subsec:code},
we fix notions and notation
about codes with components in a finite abelian group.
In Section~\ref{subsec:overlattice},
we explain how to construct an even unimodular  lattice
from a code with components in the discriminant group  $\discg_R$
of an even lattice $R$.
In Section~\ref{sec:GQRC}, 
we give the definition  of a generalized quadratic residue code,
and investigate its automorphisms.
In Section~\ref{sec:64},
we construct the lattice $L_{\QQQ}$, and 
prove that $L_{\QQQ}$ is extremal and that
$\OG(L_{\QQQ})$ contains a subgroup $\Gamma_{\QQQ}$ of order $59520$
that fits in the exact sequence~\eqref{eq:Gammaexact}.
In particular, a brute-force  method of 
the proof of $\min(L_{\QQQ})=6$
is explained in  detail.
In Section~\ref{sec:short},
we calculate the set $\SSS$ of  vectors of square-norm $6$ in $L_{\QQQ}$.
Using this set, we prove Proposition~\ref{prop:EEE8},
and  calculate the order of $\OG(L_{\QQQ})$.
In the last section,
we give another construction of  $L_{\QQQ}$.
%
\par
The computational data obtained in this article is available from the author's website~\cite{compdata}.
In particular, the Gram matrix of $L_{\QQQ}$  is found in~\cite{compdata}.
A generating set of $\OG(L_{\QQQ})$ is available from  in~\cite{NS},
though it is not minimal.
For the computation,
we used {\tt GAP}~\cite{GAP}. 
\par
\medskip
Thanks are due to Professor Masaaki Harada for informing us of the 
extremal lattices of rank $64$ in~\cite{HKO2002} and~\cite{HM2014}.
We also thank Professor Masaaki Kitazume and 
Professor Gabriele Nebe for the comments.
\par
\medskip
{\bf Conventions.}
The action of a group on a set is from the \emph{right},
unless otherwise stated.
%
%
\section{Preliminaries}\label{sec:pre}
\subsection{Codes over a finite abelian group}\label{subsec:code}
\begin{definition}
Let $A$ be a finite abelian group.
A \emph{code of length $m$ over $A$} is a subgroup of $A^m$.
\end{definition}
Let $G$ be a group.
Then the symmetric group $\SSSS_m$ acts on $G^m$ by permutations of components.
We denote by $G\wr \SSSS_m$ the wreath product 
$G^{m} \rtimes \SSSS_m$.
Then we have a splitting exact sequence
\begin{equation}\label{eq:wrexact}
1\;\to\; G^m \;\to\; G\wr \SSSS_m \;\to\; \SSSS_m \;\to\; 1.
\end{equation}
Suppose that $G$ acts on a set $X$. 
Since $\SSSS_m$ acts on $X^m$ by permutations of components
and $G^m$ acts on $X^m$ by 
$$
(x_1, \dots, x_m)^{(g_1, \dots, g_m)}=(x_1^{g_1}, \dots, x_m^{g_m}),\quad \textrm{where}\quad x_i\in X,\;\;\; g_i\in G, 
$$
the group $G\wr \SSSS_m$ acts on $X^m$ in a natural way.
\par
Let $H$ be a subgroup of the automorphism group $\Aut(A)$ of a finite abelian group $A$.
Then $H\wr \SSSS_m$ acts on  $A^m$.
For a code $\CCC$  of length $m$ over $A$,
we put
$$
\Aut_{H}(\CCC):=\set{g\in H\wr \SSSS_m}{\CCC^g=\CCC}.
$$
\subsection{Discriminant forms and overlattices}\label{subsec:overlattice}
Let $R$ be an even lattice.
We define the \emph{dual lattice} of $R$ by 
$$
R\dual:=\set{x\in R\tensor \Q}{\intf{x, v}\in \Z\;\;\textrm{for all}\;\; v\in R},
$$
and the \emph{discriminant group} $\discg_R$ of $R$  by
$$
\discg_R:=R\dual/R.
$$
Note that $R\dual$ has a natural $\Q$-valued symmetric bilinear form
that extends the $\Z$-valued symmetric bilinear form of $R$.
Hence $\discg_R$ is naturally equipped with a  quadratic form 
$$
\discf_R\colon \discg_R\to \Q/2\Z
$$
defined by $\discf_R(x \bmod R):=x^2 \bmod 2\Z$.
We call $\discf_R$  the \emph{discriminant form} of $R$.
We denote by $\OG(\discf_R)$ the automorphism group 
of the finite quadratic form $(\discg_R, \discf_R)$.
Then we have a natural homomorphism
$$
\eta_R\colon \OG(R)\to \OG(\discf_R).
$$
\begin{remark}
The notion of discriminant forms  was introduced by Nikulin~\cite{Nik}
for the study of $K3$ surfaces,
and it has been widely used in
the investigation of $K3$ surfaces and Enriques surfaces. (See, for example,~\cite{Shi}.)
\end{remark}
The discriminant form of the orthogonal direct-sum $R^m$ of $m$ copies of $R$
is the orthogonal direct-sum $(\discg_R^m, \discf_R^m)$ of $m$ copies of $(\discg_R, \discf_R)$.
Let $\CCC$ be a code of length $m$ over $\discg_R$
that is totally isotropic with respect to the quadratic form 
$$
\discf_R^m\colon \discg_R^m\to \Q/2\Z.
$$
Then the pull-back 
\begin{equation}\label{eq:LCCC}
L_{\CCC}:=\pr\inv (\CCC)
\end{equation}
of $\CCC$ by the natural projection $\pr\colon R\sp{\vee m}\to \discg_R^m$
with the restriction of the natural $\Q$-valued symmetric bilinear form of $R\sp{\vee m}$
is an even lattice that contains $R^m$ as a sublattice of finite index;
that is, $L_{\CCC}$ is an even \emph{overlattice} of $R^m$.
Moreover, since the index of $R^m$ in $L_{\CCC}$ is equal to $|\CCC|$,
if $\CCC$ satisfies 
$$
|\CCC|^2=|\discg_R|^m,
$$ 
then $L_{\CCC}$ is unimodular.
\par
Let $\CCC$ be a code of length $m$ over $\discg_R$
totally isotropic with respect to $\discf_R^m$.
We put
$$
H(R):=\Image(\eta_R\colon \OG(R)\to \OG(\discf_R))\;\;\subset \;\; \Aut(\discg_R),
$$
and consider the  group $\Aut_{H(R)}(\CCC)$.
Each element $g$ of $\Aut_{H(R)}(\CCC)$ 
is uniquely written as
$$
g=\sigma \cdot (h_1, \dots, h_m)\qquad (\;\sigma\in \SSSS_m, \;h_i\in H(R)\;).
$$
By the definition of $H(R)$,
there exist elements $\tilde{h}_i\in \OG(R)$
such that $\eta_R(\tilde{h}_i)=h_i$
for $i=1, \dots, m$.
Since $g$ preserves the code $\CCC$,
the action of
$$
\tilde{g}:=\sigma \cdot  (\tilde{h}_1, \dots, \tilde{h}_m)\;\;\in\;\;  \OG(R)\wr \SSSS_m
$$
on $(R\dual)^m$ preserves the submodule $L_{\CCC}\subset (R\dual)^m$,
and hence we obtain a lift  $\tilde{g}\in \OG(L_{\CCC})$
of $g$.
If $\eta_R$ is injective,
then the lift $\tilde{g}$ of $g$ is unique.
Therefore  we have the following:
\begin{lemma}\label{lem:AutcodetoOL}
Let $\CCC$ and  $L_{\CCC}$ be as above.
If the natural homomorphism $\eta_R$ is injective,
then 
we have an injective  homomorphism
$\Aut_{H(R)}(\CCC)\inj \OG(L_{\CCC})$.
\end{lemma}
\section{Generalized quadratic residue codes}\label{sec:GQRC}
\subsection{Definition}\label{subsec:defGQRC}
Let $A$ be a finite abelian group, and 
$p$ an odd prime.
We consider the set of rational points 
$$
\P^1(\F_p)=\F_p\cup\{\infty\}=\{0,1,\dots, p-1, \infty\}
$$
of the projective line over $\F_p$,
and let $A^{p+1}$ denote the abelian group  of all mappings 
$$
\vv\colon \P^1(\F_p)\to A
$$
from $\P^1(\F_p)$ of $A$.
Let $\chi_p\colon \F_p\sptimes\to  \{\pm 1\}$
denote the Legendre character of the multiplicative group $\F_{p}\sptimes:=\F_p\setminus\{0\}$.
\begin{definition}
Let $a,b,d,s,t,e$ be elements of $A$.
A  \emph{generalized quadratic residue code} of length $p+1$ over $A$ with parameter $(a,b,d,s,t, e)$
is the subgroup  $\QQQ$ of $A^{p+1}$ generated by the  elements
$\vv_{\infty}$, $\vv_{0}, \vv_{1}, \dots,  \vv_{p-1} \in A^{p+1}$ defined as follows:
$$
\vv_{\infty}(\nu)=\begin{cases}
a & \textrm{if $\nu \in \F_p$}, \\
b & \textrm{if $\nu=\infty$}, 
\end{cases}
$$
and, for $\mu\in \F_p$, 
$$
\vv_{\mu}(\nu)=\begin{cases}
d & \textrm{if $\nu=\mu$}, \\
s & \textrm{if $\nu\in \F_p\setminus \{\mu\}$ and  $\chi_p(\mu-\nu)=1$}, \\
t & \textrm{if $\nu\in \F_p\setminus \{\mu\}$ and  $\chi_p(\mu-\nu)=-1$}, \\
e & \textrm{if $\nu=\infty$}.
\end{cases}
$$
\end{definition}
\subsection{Automorphisms of a generalized quadratic residue code}\label{subsec:autGQRC}
Let $A$ and $p$ be as above.
For simplicity, we put  
$$
\SSSS:=\SSSS(\P^1(\F_p))\cong \SSSS_{p+1}.
$$
The linear fractional transformation embeds $\PSL_2 (p)$
into $\SSSS$.
Let $\alpha$ be a generator of $\F_p\sptimes$.
Then $\PSL_2(p)$ is generated by the three elements
$$ 
\left[\begin{array}{cc} 0 & 1 \\ -1 & 0\end{array}\right],
\quad
\left[\begin{array}{cc} 1 & 1 \\ 0 & 1\end{array}\right],
\quad
\left[\begin{array}{cc} \alpha& 0 \\  0& \alpha\inv\end{array}\right],
$$
which correspond respectively  to the permutations of $\P^1(\F_p)$ defined as follows:
$$
\xi \colon \nu\mapsto -1/\nu, 
\quad
\eta \colon  \nu\mapsto \nu+1,
\quad
\zeta \colon \nu\mapsto \alpha^2 \nu,
$$
with the understanding that $-1/0=\infty, -1/\infty=0, \infty+1=\infty$, and $\alpha^2\infty=\infty$.
%
Let $\QQQ\subset A^{p+1}$  be a generalized quadratic residue code of length $p+1$ over $A$,
and $H$ a subgroup of $\Aut(A)$.
Let $f_{\QQQ}$ be the composite homomorphism
of  the natural inclusion $\Aut_{H}(\QQQ) \inj H\wr \SSSS$ 
and the surjection $H\wr \SSSS\surj \SSSS$ in~\eqref{eq:wrexact}:
$$
f_{\QQQ} \colon \Aut_{H}(\QQQ) \inj H\wr \SSSS\surj \SSSS
$$
\begin{lemma}\label{lem:etazeta}
The image of $f_{\QQQ}$ contains $\eta$ and $\zeta$.
\end{lemma}
\begin{proof}
The permutation of components given by $\eta$ (resp.~by $\zeta$)
preserves the generating set $\{\vv_{\infty}, \vv_{0}, \dots, \vv_{p-1}\}$ of $\QQQ$.
\end{proof}
%
%
\section{An extremal lattice $L_{\QQQ}$  of rank  $64$}\label{sec:64}
We construct an extremal lattice $L_{\QQQ}$ of rank $64$.
Let $R$ be the lattice of rank $2$ with  a basis $e_1, e_2$
such that the Gram matrix of $R$ with respect to $e_1, e_2$ is
\begin{equation}\label{eq:GramR}
\left[\begin{array}{cc} \intf{e_1, e_1} & \intf{e_1, e_2} \\ \intf{e_2, e_1} & \intf{e_2, e_2}\end{array}\right]
=\left[\begin{array}{cc} 6 & 1 \\ 1 & 6 \end{array}\right].
\end{equation}
Let $e_1\dual, e_2\dual$  be the basis of  $R\dual$
dual to $e_1, e_2$.
Then $\discg_R$ is a cyclic group of order $35$ generated by
$$
u:=6 e_1\dual+2 e_2\dual=\frac{1}{35}(34 e_1+ 6 e_2).
$$
For simplicity, we denote by $n\in \Z/35\Z$ the element 
$$
n\,(6 e_1\dual+2 e_2\dual)\in \discg_R.
$$
Then the discriminant form $\discf_R\colon \discg_R\to \Q/2\Z$ is given by
$$
\discf_R(n)=6 n^2/35 \bmod 2\Z.
$$
We have
$$
\OG(\discf_R)=\set{k\in (\Z/35\Z)\sptimes}{6 k^2\equiv 6 \bmod 70}=\{\pm 1, \pm 6\}.
$$
On the other hand, the group $\OG(R)$ is of order $4$ and is 
generated by
$$
g_1:=\left[\begin{array}{cc} 0 & 1 \\ 1 & 0 \end{array}\right],
\quad
g_2:=\left[\begin{array}{cc} -1 & 0 \\ 0 & -1 \end{array}\right].
$$
The natural homomorphism $\eta_R\colon \OG(R)\to \OG(\discf_R)$
maps $g_1$ to $-6$ and $g_2$ to $-1$.
Hence $\eta_R$ is an isomorphism.
In particular, the image $H(R)$ of $\eta_R$ is isomorphic to $\Z/2\Z\times \Z/2\Z$.
\par
\medskip
We investigate 
the generalized quadratic residue code $\QQQ$ of length $32$ over $\discg_R$
with parameter 
$$
(a,b,d,s,t,e)=(0, 0, 1, 7, 3, 2).
$$
Note that  $\F_{31}\sptimes$ is generated by $3$.
We arrange the elements of $\P^1(\F_{31})$ as
\begin{equation}\label{eq:sortP1}
[\, \infty, 0 \mid 1, 3^2, 3^4, \dots, 3^{28}, \mid 3, 3^3,3^5,  \dots, 3^{29}\,],
\end{equation}
and write elements of $\discg_R^{32}$, $H(R)^{32}$,  and  $(R\tensor\Q)^{32}$
as row vectors according this arrangement.
\par
\begin{proposition}
The code $\QQQ$ is totally isotropic with respect to $\discf_R^{32}$,
and satisfies $|\QQQ|=35^{16}$.
\end{proposition}
\begin{proof}
The  code $\QQQ$ 
is generated by
the row vectors of the matrix
$\left[ I_{16}| B\right]$, where $I_{16}$ is the identity matrix of size $16$,
and $B$ is the $16\times 16$ matrix 
in Table~\ref{table:C64}.
(The components of $B$ are in $\discg_R=\Z/35\Z$.)
It is easy to confirm that $\QQQ$ is totally isotropic 
with respect to $\discf_R^{32}$, and that $|\QQQ|=35^{16}$ holds.
\begin{table}
$$
\left[\begin{array}{cccccccccccccccc} 
32 & 30 & 15 & 11 & 7 & 29 & 19 & 10 & 26 & 11 & 31 & 33 & 28 & 22 & 22 & 12 \\ 
16 & 13 & 23 & 21 & 19 & 30 & 25 & 3 & 11 & 21 & 31 & 32 & 12 & 9 & 9 & 4 \\ 
34 & 6 & 30 & 22 & 22 & 19 & 20 & 32 & 17 & 30 & 30 & 24 & 10 & 33 & 0 & 20 \\ 
34 & 26 & 1 & 17 & 9 & 6 & 4 & 17 & 14 & 12 & 25 & 19 & 34 & 8 & 33 & 20 \\ 
34 & 26 & 21 & 23 & 4 & 28 & 26 & 1 & 34 & 9 & 7 & 14 & 29 & 32 & 8 & 18 \\ 
34 & 24 & 21 & 8 & 10 & 23 & 13 & 23 & 18 & 29 & 4 & 31 & 24 & 27 & 32 & 28 \\ 
34 & 34 & 19 & 8 & 30 & 29 & 8 & 10 & 5 & 13 & 24 & 28 & 6 & 22 & 27 & 17 \\ 
34 & 23 & 29 & 6 & 30 & 14 & 14 & 5 & 27 & 0 & 8 & 13 & 3 & 4 & 22 & 12 \\ 
34 & 18 & 18 & 16 & 28 & 14 & 34 & 11 & 22 & 22 & 30 & 32 & 23 & 1 & 4 & 7 \\ 
34 & 13 & 13 & 5 & 3 & 12 & 34 & 31 & 28 & 17 & 17 & 19 & 7 & 21 & 1 & 24 \\ 
34 & 30 & 8 & 0 & 27 & 22 & 32 & 31 & 13 & 23 & 12 & 6 & 29 & 5 & 21 & 21 \\ 
34 & 27 & 25 & 30 & 22 & 11 & 7 & 29 & 13 & 8 & 18 & 1 & 16 & 27 & 5 & 6 \\ 
34 & 12 & 22 & 12 & 17 & 6 & 31 & 4 & 11 & 8 & 3 & 7 & 11 & 14 & 27 & 25 \\ 
34 & 31 & 7 & 9 & 34 & 1 & 26 & 28 & 21 & 6 & 3 & 27 & 17 & 9 & 14 & 12 \\ 
34 & 18 & 26 & 29 & 31 & 18 & 21 & 23 & 10 & 16 & 1 & 27 & 2 & 15 & 9 & 34 \\ 
34 & 5 & 13 & 13 & 16 & 15 & 3 & 18 & 5 & 5 & 11 & 25 & 2 & 0 & 15 & 29 
\end{array}\right] 
$$
\caption{The matrix $B$}\label{table:C64}
\end{table}
\end{proof}
Hence we obtain an even unimodular overlattice $L_{\QQQ}=\pr\inv(\QQQ)$ of $R^{32}$
by~\eqref{eq:LCCC}.
We will show that  $L_{\QQQ}$ is extremal, 
and that $\OG(L_{\QQQ})$ contains 
a subgroup $\Gamma_{\QQQ}$ with  the properties stated in Theorem~\ref{thm:main}.
\begin{proposition}\label{prop:eqMalgo}
The kernel of the homomorphism
$f_{\QQQ}\colon \Aut_{H(R)} (\QQQ)\to \SSSS$ 
is equal to the image of the diagonal homomorphism
$\delta\colon H(R)\inj H(R)^{32}$.
The image of $f_{\QQQ}$ contains the permutation $\xi\in \SSSS$.
\end{proposition}
\begin{proof}
Let $\sigma$ be an element of $\SSSS$.
Let $Z$ be the $32\times 32$ matrix
$$
\left[
\begin{array}{c|c}
I_{16} & B \mystrutd{4pt} \\
\hline
O & 35\, I_{16}\mystruth{11pt}
\end{array}
\right],
$$
where $B$ is regarded as a matrix with components, not in $\Z/35\Z$, but in $\Z$,
and we put 
$Z^{*}:=35 Z\inv$,
which is a matrix with components in $\Z$.
Let $Z^{\sigma}$ be the matrix obtained by applying the permutation $\sigma$ of components
to the row vectors of $Z$.
For 
$$
\vx=(x_1, \dots, x_{32}) \in H(R)^{32}\subset H(R)\wr \SSSS, 
$$
let $\Delta(\vx)$ denote the diagonal matrix with components being the representatives in $\Z$ of $x_1, \dots, x_{32}\in H(R)=\{\pm 1, \pm 6\}\subset (\Z/35\Z)\sptimes$.
Then we have $\QQQ^{\sigma\vx}=\QQQ$ only when
\begin{equation}\label{eq:cong35eq}
Z^{\sigma} \cdot \Delta(\vx)\cdot Z^{*}\equiv O \; \bmod 35.
\end{equation}
We can calculate the set
$$
\Lambda(\sigma):=\set{\gamma\in H(R)^{32}}{\QQQ^{\sigma \gamma}=\QQQ}
$$
by solving the congruence linear equation~\eqref{eq:cong35eq}
with unknowns $x_1, \dots, x_{32}$.
By this method, we obtain  
$\Lambda(\id)=\delta(H(R))$, and hence $\Ker f_{\QQQ}=\delta(H(R))$.
On the other hand, we have $\Lambda(\xi)\ne \emptyset$. 
Indeed, we see that $\Lambda(\xi)$ contains the element 
$$
(\,1, -1, \, -6, \dots, -6, \, 6, \dots, 6\,) \;\in\; H(R)^{32} \quad (\textrm{$15$ times  of $-6$ and $15$ times of $6$}).
$$
Hence $\Image f_{\QQQ}$ contains $\xi$.
\end{proof}
Combining Proposition~\ref{prop:eqMalgo} with  Lemma~\ref{lem:etazeta}, we see that 
the image of $f_{\QQQ}$ includes the subgroup $\PSL_2(31)\subset \SSSS$.
We put
$$
\overline{\Gamma}_{\QQQ}:= f_{\QQQ}\inv (\PSL_2(31)).
$$
By Lemma~\ref{lem:AutcodetoOL}, 
we have a natural embedding $\overline{\Gamma}_{\QQQ}\inj \OG(L_{\QQQ})$
of the subgroup $\overline{\Gamma}_{\QQQ}$ of $\Aut_{H(R)}(\QQQ)$ into $\OG(L_{\QQQ})$.
Let $\Gamma_{\QQQ}$ be the image of this embedding.
Then   $\Gamma_{\QQQ}$ satisfies the exact sequence~\eqref{eq:Gammaexact} in Theorem~\ref{thm:main}.
In particular, $\Gamma_{\QQQ}$ is of order $59520$.
\begin{proposition}\label{prop:minnorm}
We have $\min(L_{\QQQ})=6$.
\end{proposition}
\begin{proof}
It is easy to calculate a basis of $L_{\QQQ}$ 
and the associated Gram matrix.
Therefore the minimal norm $\min (L_{\QQQ})$ can  be  calculated by,
for example, the function {\tt ShortestVectors} of {\tt GAP}~\cite{GAP}.
However, this method did not give an answer in reasonable time.
Hence 
we adopt the following method.
\par
For $n\in \discg_R= \Z/35\Z$, we put
$$
\lambda(n):=\min \;\;\set{x^2}{ x\in R\dual, \;x \bmod R=n}.
$$
Then the values of $\lambda(n)$ are calculated as in Table~\ref{table:minlifts}.
\begin{table}
$$
\begin{array}{c|cccccccccccc}
n & 0 & \pm1 & \pm 2 & \pm 3 & \pm 4 &\pm 5 & \pm 6 &\pm 7 &\pm 8 & \pm 9 & \pm 10 \\
\hline
35 \lambda (n) & 0 & 6 & 24 & 54 & 26 & 10 & 6 & 14 & 34 & 66  & 40 
\end{array}
$$
$$
\begin{array}{c|ccccccc}
n & \pm 11 & \pm12 & \pm 13 & \pm 14 & \pm 15 &\pm 16 & \pm 17 \\
\hline
35 \lambda (n) & 26 & 24 & 34 & 56 & 90 & 66 & 54  
\end{array}\phantom{aaaaaaaaa\,\,\,}
$$
\caption{$\lambda (n)$}\label{table:minlifts}
\end{table}
For a codeword
$$
w=[\;n_{\infty}, n_{0}\; |\; n_{1}, n_{3^2}, \dots, n_{3^{28}} \; |\;  n_{3}, n_{3^3},  n_{3^5}, \dots, n_{3^{29}}\;]\;\;\in\;\;  (\Z/35\Z)^{32},
$$
we put
$$
\mu(w):=\lambda(n_{\infty})+\lambda(n_{0})+\sum_{k=0}^{14}\lambda(n_{3^{2k}})+\sum_{k=0}^{14}\lambda(n_{3^{2k+1}}).
$$
In order to prove Proposition~\ref{prop:minnorm},
it is enough to show that there exists no non-zero codeword $w$ in $\QQQ$ with $\mu(w)\le 4$.
\par
We introduce an ordering $\prec$ on $\Z/35\Z$ by
$$
m\prec m\sprime\;\; \Longleftrightarrow\;\; \tilde{m}<\tilde{m}\sprime,
$$
where $\tilde{m}\in \Z$ is the representative  of $m\in \Z/35\Z$ satisfying $0\le \tilde{m}<35$.
For $n\in \Z/35\Z$,
we denote by  $\Stab(n)$ the stabilizer subgroup of $n$ in $H(R)=\{\pm 1, \pm 6\}\subset (\Z/35\Z)\sptimes$.
Then, for each codeword $w$ of $\QQQ$,
the orbit 
$$
w^{\overline{\Gamma}_{\QQQ}}:=\set{w^\gamma}{\gamma\in \overline{\Gamma}_{\QQQ}}
$$
of $w$ under the action of  $\overline{\Gamma}_{\QQQ}$
contains at least one element 
$$
[\;n_{\infty}, n_{0}, n_{1}, n_{9}, \dots, n_{3}, n_{27}, \dots\;]
$$
with the following properties:
\begin{enumerate}[(i)]
\item $\lambda(n_{\infty})\ge \lambda(n_{\nu})$ for any $\nu\in \F_p$, 
\item $n_{\infty}\succeq k n_{\infty} $ for any $k\in H(R)=\{\pm 1, \pm 6\}$,
\item $\lambda(n_{0})\ge \lambda(n_{\nu})$ for any $\nu\in \F_p\sptimes$, 
\item $n_{0}\succeq k n_{0}$ for any $k\in \Stab(n_{\infty})$,
\item $\lambda(n_{1})\ge \lambda(n_{3^{2k}})$ for $k=1, \dots, 14$, 
and if  $\lambda(n_{1})= \lambda(n_{3^{2k}})$, then $n_{1} \succeq n_{3^{2k}}$, 
\item $n_{1}\succeq   k n_{1}$ for any $k\in \Stab(n_{\infty})\cap \Stab (n_0)$.
\end{enumerate}
By backtrack searching, 
we look for a non-zero codeword satisfying $\mu(w)\le 4$ and the properties (i)-(vi),
and confirm that there exist no such codewords in $\QQQ$.
(The arrangement~\eqref{eq:sortP1} of the points of $\P^1(\F_p)$  is convenient  for this backtrack searching.)
This task was carried out 
by distributed computation on eight CPUs of $3$ GHz.
It took us about $75$ days.
\end{proof}
Thus Theorem~\ref{thm:main} is proved,
except for the fact that $\Gamma_{\QQQ}$ is of index $2$ in $\OG(L_{\QQQ})$.  
\section{Short vectors of $L_{\QQQ}$}\label{sec:short} 
In this section,
we prove Proposition~\ref{prop:EEE8}, 
and complete the proof of Theorem~\ref{thm:main}
by showing  $|\OG(L_{\QQQ})|=119040$.
\par
\medskip
By the theory of modular forms (see, for example,~\cite[Chapter 7]{SerreBook}),
we see that the theta function of $L_{\QQQ}$ is equal to 
$$
\sum_{v\in L_{\QQQ}} q^{v^2/2}=1+2611200 \,q^3+ 19525860480 \,q^4+ 19715393260800 \,q^5 +\cdots.
$$
In particular, 
the size of the set $\SSS$ of vectors $v\in L_{\QQQ}$ of square-norm $v^2=6$
is $2611200$.
We  calculate  the set $\SSS$ and its orbit decomposition by $\Gamma_{\QQQ}$  
by the following random search method.
The result is given in Table~\ref{table:orbdecomp},
and presented more explicitly in~\cite{compdata}.
\par\medskip
{\bf Random search method.}
Let $G$ be the Gram matrix of $L_{\QQQ}$.
We set 
$$
\SSS=\{\;\}, \quad \OOO=\{\;\}.
$$
While $|\SSS|\le 2611200$,
we do the following calculation.
Let $U\in \GL_{64} (\Z)$ be a random unimodular matrix of size $64$  with integer components.
We apply the LLL  algorithm by Lenstra, Lenstra and Lov\'{a}sz~\cite{LLL} (see also~\cite[Chapter 2]{CohBook})
to 
$$
{}^U G:=U\cdot G\cdot {}^T U
$$
with the sensitivity parameter $1$.
Suppose that  we find a vector $v\sprime\in \Z^{64}$ such that $v\sprime \cdot {}^U G\cdot {}^T v\sprime =6$.
Then $v:=v\sprime \cdot U$ is a vector of square-norm $6$ in $L_{\QQQ}$.
If $v$ is not yet in $\SSS$, then we append its orbit
$o:=\shortset{v^{\gamma}}{\gamma\in \Gamma_{\QQQ}}$ to $\SSS$,
and add the set $o$ to $\OOO$.
When $|\SSS|$ reaches $2611200$,
the set $\SSS$  is equal to the set of vectors in $L_{\QQQ}$ of square-norm $6$
and $\OOO$ gives the orbit decomposition of $\SSS$ by  $\Gamma_{\QQQ}$.
\par
\medskip
\begin{table}
$$
\begin{array}{c|ccccc}
\textrm{size of an orbit} & 128  & 3968 & 11904 & 19840 & 59520 \\
\hline 
\textrm{number of orbits} & 2 & 4 & 3 & 6 & 41
\end{array}
$$
\vskip 10pt 
\caption{Orbit decomposition of $\SSS$ by $\Gamma_{\QQQ}$}
\label{table:orbdecomp}
\end{table}
\par
The set $\SSS$  is decomposed into $56$ orbits by $\Gamma_{\QQQ}$.
We choose an element $v\spar{i}$ from each orbit  $o_i$
for $i=1, \dots, 56$.
Let $\varepsilon_1, \dots, \varepsilon_8$ be the standard basis of $E_8(3)$.
We have  $\varepsilon_i^2=6$ for $i=1, \dots, 8$.
If $L_{\QQQ}$ contained a sublattice isomorphic to $E_8(3)$,
then there would exist an embedding
$$
\iota\colon \{\varepsilon_1, \dots, \varepsilon_8\}\inj \SSS
$$
that preserves the intersection form.
By the action of $\Gamma_{\QQQ}$,
we can assume that
$\iota(\varepsilon_1)$ is equal to the representative element $v\spar{i}$ of some orbit $o_i$.
By backtrack searching,
we confirm that there exists no such embedding $\iota$.
Thus Proposition~\ref{prop:EEE8} is proved.
%
%
\par
For $v\in \SSS$, we define its \emph{type} $\tau(v)$ by 
$$
\tau(v):=[\,t_0(v), \; t_1(v), \; t_2(v), \; t_3(v), \; t_6(v)\,],
$$
where $t_m(v)$ is the size of the set 
$$
\set{x\in \SSS}{\intf{x, v}=m}.
$$
Then we have $t_6(v)=1$ and 
$$
t_0(v)+2(\, t_1(v)+ t_2(v)+t_3(v) +t_6 (v)\,)=2611200
$$
for any $v\in \SSS$.
The set $\SSS$ is decomposed 
into the disjoint union
$$
\SSS\;=\;\bigsqcup\, \SSS_{\tau}, \quad \textrm{where $\SSS_{\tau}:=\shortset{v\in \SSS}{\tau(v)=\tau}$}, 
$$
according to the types,
and each $\SSS_{\tau}$ is a disjoint union of orbits $o_i$ of the action of $\Gamma_{\QQQ}$.
In Table~\ref{table:taus}, we give the list of all possible types $\tau$
and the size of each set  $\SSS_{\tau}$.
\begin{table}
{
$$
\begin{array}{ccccc|c}
t_{0}&t_{1}&t_{2}&t_{3}&t_{6}& \textrm{the size of $\SSS_{\tau}$}\\
\hline 
1368552 & 583866 & 37323 & 134 & 1 & 39680\\
1370112 & 582876 & 37503 & 164 & 1 & 39680\\
1371152 & 582216 & 37623 & 184 & 1 & 119040\\
1371880 & 581754 & 37707 & 198 & 1 & 59520\\
1372088 & 581622 & 37731 & 202 & 1 & 476160\\
1372192 & 581556 & 37743 & 204 & 1 & 119040\\
1372504 & 581358 & 37779 & 210 & 1 & 119040\\
1372608 & 581292 & 37791 & 212 & 1 & 119040\\
1372816 & 581160 & 37815 & 216 & 1 & 59520\\
1372920 & 581094 & 37827 & 218 & 1 & 158720\\
1373128 & 580962 & 37851 & 222 & 1 & 59520\\
1373232 & 580896 & 37863 & 224 & 1 & 59520\\
1373440 & 580764 & 37887 & 228 & 1 & 59520\\
1373648 & 580632 & 37911 & 232 & 1 & 119040\\
1373752 & 580566 & 37923 & 234 & 1 & 119040\\
1373960 & 580434 & 37947 & 238 & 1 & 59520\\
1374168 & 580302 & 37971 & 242 & 1 & 119040\\
1374272 & 580236 & 37983 & 244 & 1 & 59520\\
1374480 & 580104 & 38007 & 248 & 1 & 75648\\
1374584 & 580038 & 38019 & 250 & 1 & 59520\\
1374688 & 579972 & 38031 & 252 & 1 & 59520\\
1374896 & 579840 & 38055 & 256 & 1 & 178560\\
1375000 & 579774 & 38067 & 258 & 1 & 59520\\
1375104 & 579708 & 38079 & 260 & 1 & 119040\\
1376872 & 578586 & 38283 & 294 & 1 & 71424\\
1377392 & 578256 & 38343 & 304 & 1 & 23808\\
\end{array}
$$

}
\vskip 10pt 
\caption{Decomposition of   $\SSS$ by types}\label{table:taus}
\end{table}
Note that the action of $\OG(L_{\QQQ})$ preserves each $\SSS_{\tau}$.
Let $\SSS_0$ be the set  of vectors of type
$$
[\;1377392, \; 578256, \; 38343, \; 304, \; 1\;].
$$
The size $23808$ of $\SSS_0$ is minimal among all $\SSS_{\tau}$.
(See the last line of Table~\ref{table:taus}.)
This subset $\SSS_0$ is a union of two orbits $o_{k_1}$ and $o_{k_2}$ of size $11904$.
By direct calculation, we confirm the following fact:
\begin{equation}\label{eq:seven}
\parbox{10cm}{For each $v\in \SSS_0$, there exist exactly seven vectors $v\sprime$ 
in $\SSS_0$ such that $\intf{v, v\sprime}=-3$.}
\end{equation}
We find a sequence $V_0=[v_1, \dots, v_{64}]$
of vectors $v_i$ of $\SSS_0$   satisfying the following: 
\begin{enumerate}[(i)]
\item $\intf{v_i, v_j}=-3$ if and only if $|i-j|=1$, and 
\item $v_1, \dots, v_{64}$ form a basis of $L_{\QQQ}$.
\end{enumerate}
See~\cite{compdata} for the explicit vector representations of these 
vectors $v_1, \dots, v_{64}$.
We then enumerate all the sequences
$V\sprime=[v_1\sprime, \dots, v_{64}\sprime]$ 
of vectors of $\SSS_0$ such that
\begin{enumerate}[(a)]
\item $v_1\sprime$ is either $v\spar{k_1}$ or  $v\spar{k_2}$,
where $v\spar{k_{\nu}}$ is the fixed representative of the orbit $o_{k_{\nu}}$ contained in $\SSS_0$, and
\item $\intf{v_i, v_j}=\intf{v_i\sprime, v_j\sprime}$ for $i, j=1, \dots, 64$.
\end{enumerate}
Then we obtain exactly 
$10$ sequences $V_1, \dots, V_{10}$ with these properties.
Since the action of $\OG(L_{\QQQ})$ preserves $\SSS_{0}=o_{k_1}\sqcup o_{k_2}$
and the action of $\Gamma_{\QQQ}$ is transitive on each of $o_{k_1}$ and $o_{k_2}$,
we see that, for each $g\in \OG(L_{\QQQ})$,
there exists an element $h\in \Gamma_{\QQQ}$ such that
$$
V_0^{gh}=[v_1^{gh}, \dots, v_{64}^{gh}]\in \{V_1, \dots, V_{10}\}.
$$
For each $i=1, \dots, 10$,
we calculate
the matrix $g_i\in \OG(L_{\QQQ}\tensor \Q)$ such that
$V_0^{g_i}=V_i$.
It turns out that these $g_i$ preserve $L_{\QQQ}\subset L_{\QQQ}\tensor \Q$,
and hence we have $g_i\in \OG(L_{\QQQ})$.
By construction,
the group $\OG(L_{\QQQ})$ is generated by $\Gamma_{\QQQ}$ together with $g_1, \dots, g_{10}$.
We calculate the order of $|\OG(L_{\QQQ})|$.
It turns out that
$$
|\OG(L_{\QQQ})|=119040=2 |\Gamma_{\QQQ}|.
$$
Thus the proof of Theorem~\ref{thm:main} is completed.
\begin{remark}
If $g\in \OG(L_{\QQQ})$ is not contained in $\Gamma_{\QQQ}$,
then $g$ does not preserve the sublattice $R^{32}\subset L_{\QQQ}$,
and hence does not induce an automorphism of the code $\QQQ$.
\end{remark}
\section{Another construction of $L_{\QQQ}$}
Let $o_1$ and $o_2$ be the two orbits of size $128$ in $\SSS$
(see Table~\ref{table:orbdecomp}).
Let $\gen{o_1}$ and $\gen{o_2}$ be the sublattices of $L_{\QQQ}$ generated 
by $o_1$ and by $o_2$, respectively.
It is easily confirmed that both $\gen{o_1}$ and $\gen{o_2}$ are of rank $64$ and that
\begin{equation}\label{eq:orbitssum}
\gen{o_1}+\gen{o_2}=L_{\QQQ}.
\end{equation}
For simplicity, we put
$$
E:=\{e_1\spar{1}, e_2\spar{1}, \dots, e_1\spar{32}, e_2\spar{32}\},
$$
where $e_1\spar{i}$ and $e_2\spar{i}$ are the standard basis of the $i$th component  of $R^{32}$ satisfying~\eqref{eq:GramR}.
One of the two orbits of size $128$, say $o_1$,
is equal to the union of $E$ and $-E$.
For each $e_{\nu}\spar{i}\in E\subset o_1$,
there exists a unique vector $f_{+\nu}\spar{i}\in o_2$ (resp.~ $f_{-\nu}\spar{i}\in o_2$) such that
$\intf{e_{\nu}\spar{i}, f_{+\nu}\spar{i}}=2$ (resp.~ $\intf{e_{\nu}\spar{i}, f_{-\nu}\spar{i}}=-2$).
The mapping
$$
e_{1}\spar{i}\mapsto  f_{+1}\spar{i}, \quad e_{2}\spar{i}\mapsto  f_{-1}\spar{i}
$$
induces an isometry
$$
\rho\colon \gen{o_1}\isom \gen{o_2},
$$
and hence gives rise to $\rho\tensor\Q\in \OG(R^{32}\tensor\Q)$.
\begin{remark}
The orthogonal transformation  $\rho\tensor\Q$ of $R^{32}\tensor\Q$
does not preserve $L_{\QQQ}\subset R^{32}\tensor\Q$.
Indeed, the order of   $\rho\tensor\Q\in \OG(R^{32}\tensor\Q)$ is infinite.
\end{remark}
The matrix representation $M_{\rho}$ of $\rho\tensor\Q$ with respect to the basis $E$ of $R^{32}\tensor\Q$
is related to generalized quadratic residue codes   as follows.
Let
$T$ be the $32\times 32$ matrix whose rows and columns are indexed by $\P^1(\F_{31})$
sorted as in~\eqref{eq:sortP1}, 
and whose $(\mu, \nu)$th component is the string  
$$
\begin{cases}
\str{a} & \textrm{if $\mu=\infty$ and $\nu\ne \infty$},\\
\str{b} & \textrm{if $\mu=\infty$ and $\nu=\infty$},\\
\str{d} & \textrm{if $\mu=\nu\ne \infty$}, \\
\str{s} & \textrm{if $\mu\ne\infty$, $\nu\ne\infty$, $\mu\ne\nu$,  and  $\chi_{31}(\mu-\nu)=1$}, \\
\str{t} & \textrm{if $\mu\ne\infty$, $\nu\ne\infty$, $\mu\ne\nu$,  and  $\chi_{31}(\mu-\nu)=-1$}, \\
\str{e} & \textrm{if $\mu\ne\infty$ and $\nu=\infty$}\; ; 
\end{cases}
$$
that is, $T$ is the template matrix of quadratic residue codes of length $32$.
We put
\begin{eqnarray*}
&&
m_a:=\frac{1}{35}
\left[\begin{array}{cc} 
1 & -6 \\ 
6 & -1 
\end{array}\right],
\;\;\;
m_b:=\frac{1}{35}
\left[\begin{array}{cc} 
12 & -2 \\ 
2 & -12 
\end{array}\right],
\;\;\;
m_d:=\frac{1}{35}
\left[\begin{array}{cc} 
12 & -2 \\ 
2 & -12 
\end{array}\right],
\\
&&
m_s:=\frac{1}{35}
\left[\begin{array}{cc} 
-6 & 1 \\ 
-1 & 6 
\end{array}\right],
\;\;\;
m_t:=\frac{1}{35}
\left[\begin{array}{cc} 
6 & -1 \\ 
1 & -6 
\end{array}\right],
\;\;\;\;\;\;\;
m_e:=\frac{1}{35}
\left[\begin{array}{cc} 
-1 & 6 \\ 
-6 & 1 
\end{array}\right].
\end{eqnarray*}
\begin{proposition}
The matrix representation $M_{\rho}$ of $\rho\tensor\Q$ with respect to the basis $E$ of $R^{32}\tensor\Q$
is obtained from the template matrix $T$ by substituting
$\str{a}$ with $m_a$,
$\str{b}$ with $m_b$,
$\str{d}$ with $m_d$,
$\str{s}$ with $m_s$,
$\str{t}$ with $m_t$, and 
$\str{e}$ with $m_e$.
\end{proposition}
By~\eqref{eq:orbitssum},
we obtain another method of construction of $L_{\QQQ}$ as follows.
\begin{proposition}
The lattice $L_{\QQQ}$ is generated by $E$ and $E^{\rho}$ in $R^{32}\tensor\Q$ .
\end{proposition}
\par
\bigskip
{\bf Note added on 2018/05/04:}
Masaaki Harada confirmed $\min(L_{\QQQ})=6$ by a direct computation using {\tt Magma}.
It took about 27 days.
We thank Professor Masaaki Harada for this heavy computation.
\bibliographystyle{plain}

%
%

%
%
\end{document}